\documentclass[final]{siamltex}

\usepackage{amsmath}
\usepackage{amssymb}
\usepackage{graphicx}
\newtheorem{remark}[theorem]{Remark} 

\newcommand{\ee}{\text{\rm e}}
\newcommand{\wt}{\widetilde}

\title{Convergence analysis of Strang splitting for Vlasov-type equations
}

\newcommand*\samethanks[1][\value{footnote}]{\footnotemark[#1]}
\author{
Lukas Einkemmer\thanks{Department of Mathematics, University of Innsbruck, Technikerstra\ss e 13, Innsbruck, Austria ({\tt lukas.einkemmer@uibk.ac.at}, {\tt alexander.ostermann@uibk.ac.at}).
The first author was supported by a scholarship of the Vizerektorat f\"ur Forschung, University of Innsbruck, and by the Austrian Science Fund (FWF), project id: P25346.} \and Alexander Ostermann\samethanks
}

\begin{document}

\maketitle

\begin{abstract}
A rigorous convergence analysis of the Strang splitting algorithm for Vlasov-type equations in the
setting of abstract evolution equations is provided. It is shown that under suitable assumptions the convergence is of second order in the time step $\tau$. As an example, it is verified that the
Vlasov--Poisson equations in 1+1 dimensions fit into the framework of this analysis. Further, numerical experiments for the latter case are presented.
\end{abstract}



\pagestyle{myheadings}
\thispagestyle{plain}
\markboth{L.~EINKEMMER AND A.~OSTERMANN}{CONVERGENCE ANALYSIS OF STRANG SPLITTING}

\section{Introduction}

The most fundamental theoretical description of a (collisionless) plasma comes from the kinetic equation. This so called Vlasov equation is given by (see e.g.~\cite{Belli:2006}) 
\begin{equation*}
	\partial_t f(t,\boldsymbol{x},\boldsymbol{v}) + \boldsymbol{v} \cdot \nabla_{\boldsymbol{x}} f(t,\boldsymbol{x},\boldsymbol{v}) +
	\boldsymbol{F} \cdot \nabla_{\boldsymbol{v}} f(t,\boldsymbol{x},\boldsymbol{v})=0,
\end{equation*}
where $\boldsymbol{x}$ denotes the position and $\boldsymbol{v}$ the velocity.
The function $f$ describes a particle-probability distribution in the $3+3$ dimensional phase space. Since a plasma interacts with the electromagnetic field in a non-trivial manner, the Vlasov equation needs to be coupled to the electromagnetic field through the force term $\boldsymbol{F}$. A one-dimensional example is given in section~\ref{sec:vp} below.

Depending on the application, either the full Vlasov--Maxwell equations or a simplified model is appropriate. Such models include, for example, the Vlasov--Poisson and the gyrokinetic equations.

Due to the high dimensionality of the equations the most common numerical approach are so
called particle methods. In this class of methods, the phase space is left to be continuous and a (large) number of particles with various starting points are advanced in time. This is possible due to the structure of the equations, which implies that a single particle evolves along a trajectory given by an ordinary differential equation. A number of such methods have been developed, most notably the particle-in-cell (PIC) method. Such methods have
been extensively used for various applications (see e.g.~\cite{Fahey:2008}). The PIC scheme gives reasonable
results in case where the tail of the distribution is negligible. If this is not the case the method suffers from numerical noise that only decreases as ${1}/{\sqrt{n}}$, where $n$ denotes the number of particles (see e.g.~\cite{Heath:2011} or \cite{Filbet:2001}). Motivated by these considerations, a number of schemes employing discretization in phase space have been proposed. A comparison of various such methods can be found in \cite{Filbet:2001}.

Using a time splitting scheme for the Vlasov--Poisson equations was first proposed by \cite{Cheng:1776} in 1976. In \cite{Mangeney:2002} the method was extended to the Vlasov--Maxwell equations. In both cases, second-order Strang splitting (see e.g.~\cite{Jahnke00}) is used to advance the solution of the Vlasov equation in time.

Quite a few convergence results are available for semi-Lagrangian methods that employ Strang splitting. For example, in \cite{Besse:2005}, \cite{Besse:2008} and \cite{Respaud:2011} convergence is shown in the case of the 1+1 dimensional Vlasov--Poisson equations. Both \cite{Besse:2005} and \cite{Besse:2008} assume the same analytical framework, regarding the regularity of the solution, that we employ in section \ref{sec:vp}. However, the convergence proofs presented in these papers are based on the method of characteristics and are valid only if certain assumptions are made, which hold for the Vlasov--Poisson equations in combination with the specific scheme under consideration in those papers. This is in contrast to our analysis,  as we, for example, do not limit ourselves to a specific form of the auxiliary method (the technical details of this will be apparent in section \ref{sec:the_strang_splitting_algo}). The resulting convergence results for the Vlasov--Poisson equations, however, are similar to what we derive in section \ref{sec:vp}. 
Furthermore, the convergence of a special case of the one-dimensional Vlasov--Maxwell equation in the laser-plasma interaction context is investigated in~\cite{Bostan:2009}.

In this paper, we will consider a class of Vlasov-type equations as abstract evolution equation (i.e., without discretization in space).
In this context we will derive sufficient conditions such that the Strang splitting algorithm is convergent of
order $2$. We will then verify these conditions for the example of the Vlasov--Poisson equations in
1+1 dimensions and present some numerical results.

\section{Setting}

We will investigate the following (abstract) initial value problem
\begin{equation} \label{eq:abstract_ivp}
	\left\{
	\begin{aligned}
		f'(t)&=(A+B)f(t)  &  \\
		f(0)&=f_{0}.   \\
	\end{aligned}
	\right.
\end{equation}
We assume that $A$ is an (unbounded) linear operator and that the non-linearity $B$ has the form $Bf=B(f)f$, where $B(f)$ is an (unbounded) linear operator. We will consider this abstract initial value problem on a finite time interval $[0,T]$.

Problem~\eqref{eq:abstract_ivp} comprises the Vlasov--Poisson and the Vlasov--Maxwell equations for $A = -\boldsymbol{v}\cdot \nabla_{\boldsymbol{x}}$ and appropriately chosen $B$ as special cases. It is also general enough to include the gyrokinetic equations (as stated, for example, in \cite{Hahm:2009}).
The Vlasov--Poisson equations are considered in more detail in section \ref{sec:vp}.

\subsection{The Strang splitting algorithm} \label{sec:the_strang_splitting_algo}

Let $f_k\approx f(t_k)$ denote the numerical approximation to the solution of \eqref{eq:abstract_ivp} at time $t_k=k\tau$ with step size $\tau$. We assume that the differential equations $f' = Af$ and $g' = B_{k+1/2}g$, where $B_{k+1/2}$ is a suitable approximation to the operator $B\left(f(t_k+\frac{\tau}{2})\right)$, can be solved efficiently. In this paper we always make the choice $B_{k+1/2}=B(f_{k+1/2})$, where
\begin{equation}\label{eq:midpoint}
f_{k+1/2} = \Psi(\tfrac{\tau}2, f_k)
\end{equation}
is a first-order approximation to the solution of \eqref{eq:abstract_ivp} at time $t=t_k+\frac{\tau}2$. Note that $f_{k+1/2}$ typically depends on $f_k$ only. In the case of the Vlasov--Poisson equations, an appropriate choice is
\begin{equation*}
f_{k+1/2} = \ee^{\frac{\tau}2 B(f_k)}\ee^{\frac{\tau}{2}A}f_k
\end{equation*}
or even $f_{k+1/2} = \ee^{\frac{\tau}{2}A}f_k$, as will be explained in the first paragraph of section \ref{sec:implementation}.

The idea of Strang splitting is to advance the numerical solution by the recursion $f_{k+1} = S_k f_k$, where the (nonlinear) splitting operator $S_k$ is given by
\begin{equation}\label{eq:def-strang-one}
	S_k = \ee^{\frac{\tau}{2}A}\ee^{\tau B_{k+1/2}}\ee^{\frac{\tau}{2}A}.
\end{equation}
The precise conditions on $f_{k+1/2}$ for proving convergence are given in section~\ref{sec:convergence_abstract} below. Resolving this recursion, we can compute an approximation to the exact solution at time $T$ by
\begin{equation}\label{eq:def-strang}
	f_n = \left( \prod_{k=0}^{n-1} S_{k} \right) f_0 = S_{n-1}\cdots S_0f_0,
\end{equation}
where $n$ is an integer chosen together with the step size $\tau$ such that $T=n\tau$.

\subsection{Preliminaries}

For the convenience of the reader we collect some well known results that are used quite extensively
in section \ref{sec:convergence_abstract}.

To bound the remainder term $R_k(f)$ of a Taylor expansion
\begin{equation*}
	f(\tau) = f(0)+ \tau f'(0)+\ldots +\frac{\tau^{k-1}}{(k-1)!} f^{(k-1)}(0) + \tau^{k} R_k(f),
\end{equation*}
we will use the integral form
\begin{equation*}
	R_k(f) = \frac{1}{(k-1)!}\int_0^1 f^{(k)}(\tau s)(1-s)^{k-1}\,\mathrm{d}s,
\end{equation*}
where $k\geq 1$. Note that it is implicitly understood that $R_k$ is a function of $\tau$ as well. However, since we will work mostly with a fixed $\tau$, it is convenient to drop it in the notation of $R_k$. For
convenience we also define
\begin{equation*}
    R_0(f)=f(\tau).
\end{equation*}

For (unbounded) linear operators it is more convenient to work with the $\varphi$ functions instead of
the remainder term given above.

\begin{definition} [{\rm $\varphi$ functions}]  \label{def:varphi}
	Suppose that the linear operator $E$ generates a $\mathcal{C}_0$ semigroup. Then we define the bounded operators
	\begin{equation}
		\begin{aligned} \label{eq:varphi}
			\varphi_0(\tau E) &= \ee^{\tau E}, \\
			\varphi_k(\tau E) &= \int_0^1 \ee^{(1-\theta)\tau E}\frac{\theta^{k-1}}{(k-1)!} \,\mathrm{d}\theta \quad \text{for}\ k\geq 1.
		\end{aligned}
	\end{equation}
\end{definition}

Since we are merely interested in bounds of such functions, we will never directly employ the definition
given. Instead we will work exclusively with the following recurrence relation.

\begin{lemma}
	The $\varphi$ functions satisfy the recurrence relation
	\begin{equation}\label{eq:phi-rec}
		\varphi_k(\tau E) = \frac{1}{k!} + \tau E \varphi_{k+1}(\tau E), \quad k\ge 0
	\end{equation}
	and in particular (for $\ell\in\mathbb{N}$)
	\begin{equation*}
		\ee^{\tau E} = \sum_{k=0}^{\ell-1} \frac{\tau^k}{k!}E^k + \tau^\ell E^\ell \varphi_\ell(\tau E).
	\end{equation*}	
\end{lemma}
\begin{proof}
	The first relation follows from integration by parts applied to \eqref{eq:varphi}. The second one results from using $\varphi_0 = \ee^{(\cdot)}$ and applying the first relation repeatedly.
\end{proof}

The $\varphi$ functions are used to expand the exponential of some linear operator. In the sense of the previous lemma, these functions play the same role for an exponential of a linear operator as does the remainder term in Taylor's theorem.

Suppose that the differential equation $g' = G(g)$ has (for a given initial value) a unique solution. In this case we denote the solution at time $t$ with initial value $g(t_0)=g_0$ with the help of the evolution operator, i.e. $g(t)=E_G(t-t_0,g_0)$.

The Gr\"obner--Alekseev formula (also called the nonlinear variation-of-constants formula) will be employed quite heavily.

\begin{theorem} [{\rm Gr\"obner--Alekseev formula}] Suppose that there exists a unique $f$ satisfying
\begin{equation*}
	\left\{
	\begin{aligned}
		f^{\prime}(t) &= G(f(t))+R(f(t)) \\
		f(0)          &= f_0
	\end{aligned}
	\right.
\end{equation*}
and that $g' = G(g)$ has (for a given initial value) a unique solution. Then it holds that
	\begin{equation*}
		f(t)=E_{G}(t,f_{0})+\int_{0}^{t}\partial_{2}E_{G}\left(t-s,f(s)\right)R\left(f(s)\right)\, \mathrm{d}s.
	\end{equation*}
\end{theorem}
\begin{proof}
    For linear (and possibly unbounded) $G$, this formula is proved in \cite{Holden2012} by the fundamental theorem of calculus. Here, we prove the extension to nonlinear $G$. Let us assume that $u(t)$ is a solution of $u^{\prime}(t)=G\left(u(t)\right)$. By differentiating
	\begin{equation*}
		E_{G}\left(t-s,u(s))\right)=u(t)
	\end{equation*}
	with respect to $s$ we get
	\begin{equation*}
		 -\partial_{1}E_{G}\left(t-s,u(s)\right)+\partial_{2}E_{G}\left(t-s,u(s)\right)G\left(u(s)\right)=0.
	\end{equation*}
	The initial value of $u$ is now chosen such that $u(s)=f(s)$ which implies
	\begin{equation*}
		 -\partial_{1}E_{G}\left(t-s,f(s)\right)+\partial_{2}E_{G}\left(t-s,f(s)\right)G\left(f(s)\right)=0.
	\end{equation*}
	Altogether we have for $\psi(s)=E_G(t-s,f(s))$ (by the fundamental theorem of calculus)
	\begin{align*}
		f(t) - E_G(t,f_0) &= \int_0^t \psi'(s)\,\mathrm{d}s \\
        &=\int_0^t\Bigl(-\partial_1 E_G(t-s,f(s))+ \partial_2 E_G(t-s,f(s))f'(s)\Bigl)\,\mathrm{d}s\\
        &= \int_0^t \partial_{2}E_{G}\left(t-s,f(s)\right)R\left(f(s)\right)\,\mathrm{d}s,
	\end{align*}
	as desired.
\end{proof}

Since anticommutator relations appear quite naturally in some expansions, we will employ the notation
\begin{equation*}
	\left\{ E_1, E_2 \right\} = E_1E_2+E_2E_1,
\end{equation*}
for linear operators $E_1$ and $E_2$ (on a suitable domain).

In what follows $C$ will denote a generic constant that may have different values at different occurrences.

\section{Convergence analysis in the abstract setting} \label{sec:convergence_abstract}

The problem of splitting an evolution equation into two parts, governed by linear and possibly unbounded operators, has already been investigated in some detail. In \cite{HanO2008} it is shown that splitting methods with a given classical order retain this order in the stiff case (under suitable regularity assumptions).

An alternative analysis for Strang splitting in the linear case is given in \cite{Jahnke00}. The approach presented there is more involved, however, it demands less regularity on the solution. The purpose of this section is to extend this analysis to the abstract initial value problem given by \eqref{eq:abstract_ivp}.

\subsection{Convergence}

Our convergence proof will be carried out in an abstract Banach space $X$ with norm $\|\cdot\|_X$. It relies on the classical concepts of consistency and stability. We begin by stating a suitable notion of consistency for our splitting operator. For this purpose, let
\begin{equation*}
\wt B_{k+1/2} = B\left(f(t_k+\tfrac{\tau}2)\right)
\end{equation*}
denote the non-linearity, evaluated at the exact solution at time $t_k+\frac{\tau}2$. With the help of this operator, we consider the auxiliary scheme
\begin{equation*}
	\wt S_k = \ee^{\frac{\tau}{2}A}\ee^{\tau \wt B_{k+1/2}}\ee^{\frac{\tau}{2}A}.
\end{equation*}
We are now in the position to define consistency for our numerical method.

\begin{definition} [{\rm Consistency of order $p$}]
The Strang splitting algorithm \eqref{eq:def-strang-one} is consistent of order $p$ if
\begin{equation}\label{eq:bd-cons}
\Vert f(t_k+\tau)-\wt S_kf(t_k) \Vert_X \leq C \tau^{p+1}.
\end{equation}
The constant $C$ depends on the considered problem but is independent of $\tau$ and $k$ for $0\le t_k = k\tau \le T$.
\end{definition}

Note that for algorithm \eqref{eq:def-strang-one}, the order of consistency is not necessarily $p=2$. The actual order depends on the properties of the involved operators, and order reduction can occur even in the linear case, see~\cite{Jahnke00}.

To estimate the global error, i.e. $f_{k+1}-f(t_{k+1})$, we employ the error recursion
\begin{equation}\label{eq:error-recu}
f_{k+1}-f(t_{k+1}) = S_k\bigl(f_k-f(t_k)\bigr) + (S_k-\wt S_k)f(t_k) + \wt S_k f(t_k) - f(t_{k+1}).
\end{equation}
The first two terms on the right-hand side of \eqref{eq:error-recu} are controlled by the linear and non-linear stability properties of the method, whereas the last difference is controlled by the consistency bound. For our abstract convergence result, we have to assume the stability bound
\begin{equation}\label{eq:stability}
\Vert S_k \Vert_{X\gets X} \leq 1+C\tau
\end{equation}
and the Lipschitz condition
\begin{equation}\label{eq:L-cond}
\Vert S_k -\wt S_k\bigl\Vert_{X\gets X} \leq C\tau\|f_{k+1/2}-f(t_k+\tfrac{\tau}2)\bigr\|_X
\end{equation}
with a constant $C$ that is uniform in $\tau$ and $k$ for $0\le t_k = k\tau\le T$. These bounds will be verified in section~\ref{subsec:stability} for the particular case of the Vlasov--Poisson equations.

We are now in the position to bound the global error.

\begin{theorem} [{\rm Convergence}] \label{thm:convergence-abstract}
Suppose that the Strang splitting scheme \eqref{eq:def-strang-one} is consistent of order $p$ and satisfies the  bounds~\eqref{eq:stability} and~\eqref{eq:L-cond}. Further assume that the auxiliary method~\eqref{eq:midpoint} is consistent of order~$p-1$ and (locally) Lipschitz continuous with respect to its second argument. Then the Strang splitting scheme~\eqref{eq:def-strang-one} is convergent of order~$p$, i.e.
\begin{equation}
\|f_k-f(t_k)\|_X\le C\tau^p
\end{equation}
with a constant $C$ that is independent of $\tau$ and $k$ for $0\le t_k = k\tau\le T$.
\end{theorem}

\begin{proof}
The proof is quite standard. We apply the triangle inequality to the error recursion~\eqref{eq:error-recu} and insert the bounds~\eqref{eq:stability}, \eqref{eq:L-cond}, and the consistency bound~\eqref{eq:bd-cons}. By our assumptions on method \eqref{eq:midpoint}, we further obtain
\begin{align*}
\|f_{k+1/2}-f(t_k+\tfrac{\tau}2)\|_X &= 
\|\Psi(\tfrac{\tau}2,f_k) - \Psi(\tfrac{\tau}2,f(t_k))+ \Psi(\tfrac{\tau}2,f(t_k)) - f(t_k+\tfrac{\tau}2)\|_X\\
&\le C\|f_k-f(t_k)\|_X + C\tau^p.
\end{align*}
This finally results in the recursion
\begin{equation*}
\|f_{k+1}-f(t_{k+1})\|_X \le (1+C\tau)\|f_k-f(t_k)\|_X + C\tau^{p+1}
\end{equation*}
which is easily solved. Employing $f_0=f(0)$ we obtain the desired bound.
\end{proof}

\subsection{Consistency}

It is the purpose of this section to formulate assumptions under which the consistency bound~\eqref{eq:bd-cons} holds for the abstract initial value problem \eqref{eq:abstract_ivp}. To make the derivations less tedious we will adhere to the notation laid out in the following remark.

\begin{remark} \label{rem:notation}
In this section we will denote the solution of \eqref{eq:abstract_ivp} at a fixed time $t_k$ by $f_0$. The notation $f(s)$ is then understood to mean $f(t_k+s)$. The function $f_0$ is a (possible) initial value for a single time step (i.e., a single application of the splitting operator $S_k$). It is not, in general, the initial value of the solution to the abstract initial value problem as in the previous sections. If we assert that a bound holds uniformly in $t_k$, it is implied that it holds for all $f_0$ in the sense defined here (remember that $t_k\in[0,T]$). Since $t_k$ is fixed we will use the notation $\wt B$ and $\wt S$ instead of $\wt B_{k+1/2}$ and $\wt S_k$, respectively.
\end{remark}

Let us start with expanding the exact solution by using the Gr\"obner--Alekseev formula (this has been proposed in the context of the nonlinear Schr\"odinger equation in \cite{lubich:2008}). We consider the linear operator $A$ as a perturbation of the differential equation given by the non-linear operator $B$. This choice is essential for the treatment given here, since it allows us to apply the expansion sequentially without any additional difficulties.

\begin{lemma} [{\rm Expansion of the exact solution}] \label{le:exactsol}
The exact solution of \eqref{eq:abstract_ivp} has the formal expansion
\begin{equation*}
\begin{aligned}
f(\tau) &= E_{B}(\tau,f_{0})+\int_{0}^{\tau}\partial_{2}E_{B}(\tau-s,f(s))AE_{B}(s,f_{0})\,\mathrm{d}s  \\
&\quad + \int_{0}^{\tau}\int_{0}^{s}\partial_{2}E_{B}(\tau-s,f(s))A\partial_{2}E_{B}(s-\sigma,f(\sigma))
AE_{B}(\sigma,f_{0})\,  \mathrm{d}\sigma \mathrm{d}s \\
&\quad +  \int_{0}^{\tau}\int_{0}^{\sigma_1}\int_{0}^{\sigma_2}
\left( \prod_{k=0}^2 \partial_2 E_B(\sigma_k-\sigma_{k+1},f(\sigma_{k+1}))A \right) f(\sigma_3) \, \mathrm{d}\sigma_3 \mathrm{d}\sigma_2 \mathrm{d}\sigma_1,
\end{aligned}
\end{equation*}
where $\sigma_0=\tau$.
\end{lemma}

\begin{proof}	
Apply the Gr\"obner--Alekseev formula three times to equation \eqref{eq:abstract_ivp}.
\end{proof}

Next we expand the splitting operator $\wt S$ in a form that is suitable for comparison with the exact solution.

\begin{lemma} [{\rm Expansion of the splitting operator}] \label{le:approxsol}
The splitting operator $\wt S$ has the formal expansion
\begin{equation*}
\wt Sf_{0} = \ee^{\tau \wt B}f_0+\frac{\tau}{2}\left\{ A,\ee^{\tau\wt B}\right\} f_{0}+
\frac{\tau^{2}}{8}\left\{ A,\left\{ A,\ee^{\tau\wt B}\right\} \right\} f_{0}+R_{3}f_{0},
\end{equation*}
where
\begin{equation*}
R_{3}=\frac{\tau^{3}}{16}\int_{0}^{1}\left(1-\theta\right)^{2} \left\{ A,\left\{ A,\left\{ A,\ee^{\frac{\tau\theta}{2}A}\ee^{\tau\wt B}\ee^{\frac{\tau\theta}{2}A}\right\} \right\} \right\} \, \mathrm{d}\theta.
\end{equation*}
\end{lemma}

\begin{proof}
Let us define the function $g(s)=\ee^{\frac{1}{2}sA}\ee^{\tau\wt B}\ee^{\frac{1}{2} sA}$.
The first three derivatives of $g$ are given by
\begin{eqnarray*}
g^{\prime}(s) &=& \frac{1}{2} \left\{A,g(s)\right\}, \\
g^{\prime \prime}(s) &=& \frac{1}{4} \left\{A,\left\{A,g(s)\right\}\right\}, \\
g^{(3)}(s) &=& \frac{1}{8} \left\{A,\left\{A,\left\{A,g(s)\right\}\right\}\right\}.
\end{eqnarray*}
From the observation that $\wt S=g(\tau)$ and by Taylor's theorem we obtain the result.
\end{proof}

Let us now give the conditions which, if satisfied, imply that the Strang splitting scheme, in our abstract setting, is consistent of order two.

\begin{theorem} [{\rm Consistency}] \label{thm:abstract_consistency}
Suppose that the estimates

\begin{eqnarray}
	\left\Vert \varphi_{1}^{\delta_{i1}}(\wt B) \left(B(E_{B}(\tfrac{\tau}{2},f_0))-\tilde{B}\right)R_{1}^{\delta_{i0}}\left(E_{B}(\cdot,f_{0})\right)  \right\Vert_X
	&\leq& C \tau, \quad i \in \{0,1\}  \label{eq:ass_-1} \\
	\sup_{0\leq s \leq \tau} 
	\left\Vert
		\frac{\mathrm{d}^2}{\mathrm{d}s^2}  
		\ee^{s\wt B}\big(B(E_B(s,f_0))-\wt B \big)u(s) 	
	\right\Vert_X
 &\leq& C, \label{eq:ass_-0.5} \\
	\left\Vert \left[ B\left(E_B(\tfrac{\tau}{2},f_0)\right)-\wt B + \tfrac{\tau}{2} B^{\prime}(Af_0) \right] f_0 \right\Vert_X 
&\leq& C \tau^2, \label{eq:ass_0}
\end{eqnarray}
and
\begin{eqnarray}
\sup_{0\le s\le\tau}\left\Vert A^i \ee^{(\tau-s)\wt B}(B-\wt B)E_B(s,f_0) \right\Vert_X & \leq & C\tau^{2-i}, \quad i\in\{1,2\} \label{eq:ass_1} \\
\left\Vert  (B(f_0)-\wt B)Af_0 \right\Vert_X & \leq & C\tau, \label{eq:ass_2} \\
\left\Vert A^{\delta_{i2}} \wt B^{1+\delta_{i0}}\varphi_{1+\delta_{i0}}(\tau\wt B)A^{1+\delta_{i1}}f_0 \right\Vert_X & \leq & C, \quad i\in\{0,1,2\} \label{eq:ass_3} \\
\left\Vert A^{\delta_{i2}} R_{1+\delta_{i0}}(\partial_2 E_B(\cdot,f_0))A^{1+\delta_{i1}}f_0 \right\Vert_X & \leq & C, \quad i\in\{0,1,2\} \label{eq:ass_4}
\end{eqnarray}		
hold uniformly in $t$, where $\delta_{ij}$ denotes the Kronecker delta. In addition, suppose that the estimates
\begin{eqnarray}
\sup_{0\le s\le \tau}\left\Vert\frac{\mathrm{d^2}}{\mathrm{d}s^2}\Bigl( \partial_2 E_B(\tau-s, f(s))A E_B(s,f_0)\Bigr) \right\Vert_X& \leq & C, \label{eq:ass_5a} \\
\sup_{0\le\sigma\le s\le\tau}\left\Vert\frac{\partial}{\partial s}\Bigl(\partial_{2}E_{B}(\tau-s,f(s))A\partial_{2} E_{B}(s-\sigma,f(\sigma))A E_{B}(\sigma,f_{0})\Bigr)\right\Vert_X & \le & C, \label{eq:ass_5b} \\
\sup_{0\le\sigma\le s\le\tau}\left\Vert\frac{\partial}{\partial \sigma}\Bigl(\partial_{2}E_{B}(\tau-s,f(s))A \partial_{2} E_{B}(s-\sigma,f(\sigma))A E_{B}(\sigma,f_{0})\Bigr)\right\Vert_X & \le & C, \qquad\label{eq:ass_5c} \\
\biggl\Vert\biggl( \prod_{k=0}^2 \partial_2 E_B(\sigma_k-\sigma_{k+1},f(\sigma_{k+1}))A \biggr) f(\sigma_3) \biggr\Vert_X & \leq & C \label{eq:ass_6}, \\
\sup_{0\le s\le \tau}\left\Vert \left\{ A,\left\{ A,\left\{ A,e^{\frac{s}{2}A}e^{\tau B}e^{\frac{s}{2}A}\right\} \right\} \right\}f_0 \right\Vert_X & \leq & C, \label{eq:ass_7}
\end{eqnarray}
hold uniformly in $t$ for $0\le \sigma_3\le \sigma_2\le \sigma_1\le \sigma_0 = \tau$.

Then the Strang splitting \eqref{eq:def-strang-one} is consistent of order 2.
\end{theorem}

\begin{proof}
We have to compare terms of order $0$, $1$, and $2$ in Lemma \ref{le:exactsol} and Lemma~\ref{le:approxsol} and show that the remaining terms of order $3$ can be bounded as well.

%
%

\emph{Terms of order $0$}.
We have to bound the difference
\begin{equation}\label{eq:zerodiff}
\ee^{\tau\wt B}f_{0}-E_{B}(\tau,f_{0}).
\end{equation}
For this purpose we denote $E_B(s,f_0)$ by $u(s)$ and make use of the fact that $u$ satisfies the differential equation
\begin{equation*}
u^{\prime} = \wt B u+(B-\wt B)u
\end{equation*}
with initial value $f_0$. Employing the variation-of-constants formula we get
\begin{equation*}
u(\tau) = \ee^{\tau \wt B}f_0 + \int_0^\tau \ee^{(\tau-s)\wt B}(B-\wt B)E_B(s,f_0)\,\mathrm{d}s.
\end{equation*}

Now let us employ the midpoint rule; this yields
\begin{align*}
	u(\tau)-\ee^{\tau\wt{B}}f_{0} &= 
	\tau \ee^{\frac{\tau}{2}\wt B}\left(B(u(\tfrac{\tau}{2}))-B(f(\tfrac{\tau}{2}))\right)u(\tfrac{\tau}{2})+d \\
	&=
	\tau\left(B(u(\tfrac{\tau}{2}))-B(f(\tfrac{\tau}{2}))\right)f_0
	+\tfrac{\tau^2}{2}\varphi_1(\wt B)\left(B(E_{B}(\tfrac{\tau}{2}))-\wt B\right)	\\
	&\qquad+\tfrac{\tau^2}{2}\left(B(E_{B}(\tfrac{\tau}{2}))-\wt B\right)R_{1}\left(E_{B}(\cdot,f_{0})\right)
	+ d.
\end{align*}
The second term is bounded by assumption \eqref{eq:ass_-1} and the remainder term by assumption \eqref{eq:ass_-0.5}. We postpone the discussion of the first term until we have considered the terms of order 1.

%
%

\emph{Terms of order $1$}.
For
\begin{equation*}
g(s)=\ee^{(\tau-s)\wt B}A\ee^{s\wt B}f_0,\qquad k(s)=\partial_{2}E_{B}(\tau-s,f(s))AE_{B}(s,f_0)
\end{equation*}
we get (by use of the trapezoidal rule)
\begin{multline*}
\frac{\tau}{2}\Bigl(g(0)+g(\tau)\Bigr)-\int_{0}^{\tau}k(s)\, \mathrm{d}s \\
= \frac{\tau}{2}\Bigl(g(0)-k(0)+g(\tau)-k(\tau)\Bigr)-\frac{\tau^{3}}{2}\int_{0}^{1}\theta(1-\theta)k''(\theta \tau)\, \mathrm{d}\theta.
\end{multline*}
First, let us compare $g(\tau)$ and $k(\tau)$
\begin{equation*}
g(\tau)-k(\tau)= A\bigl(\ee^{\tau\wt B}f_{0}-E_{B}(\tau,f_{0})\bigr),
\end{equation*}
which is the same term that we encountered in \eqref{eq:zerodiff}, except that we have an additional $A$ to the left of the expression. We thus can apply assumption \eqref{eq:ass_1} with $i=1$. Second, we have to compare $g(0)$ and $k(0)$
\begin{equation*}
g(0)-k(0)=\bigl(\ee^{\tau \wt B}-\partial_2 E_{B}(\tau,f_{0})\bigr)Af_{0}.
\end{equation*}
Expanding both terms
\begin{align*}
\ee^{\tau \wt B} &= I + \tau\wt B +\tau^2\wt B^2 \varphi_2(\tau\wt B)\\
E_{B}(\tau,f_{0}) &= f_0 + \tau B f_0 + \tau^2 R_2(E_B(\cdot, f_0)),
\end{align*}
we get

\begin{align*}
	g(0)-k(0) &= -\tau B^{\prime}(Af_0)f_0  -\tau\bigl( B(f_0)-\wt B \bigr) Af_0  \\
		  &\qquad+ \tau^2 \left( \wt B^{2}\varphi_{2}(\tau\wt B)- R_{2}(\partial_{2} E_{B}(\cdot,f_0))\right)A f_{0}.
\end{align*}
The first term is bounded by assumption \eqref{eq:ass_0} and the second term by assumption \eqref{eq:ass_2}. The third term is
bounded by assumption \eqref{eq:ass_3} with $i=0$ and the fourth term by assumption \eqref{eq:ass_4}
with $i=0$.

Finally, we have to estimate the remainder term of the quadrature rule which is bounded by assumption~\eqref{eq:ass_5a}.

%
%

\emph{Terms of order $2$}.
For the functions
\begin{align*}	
g(s,\sigma)&=\ee^{(\tau-s)\wt B}A\ee^{(s-\sigma)\wt B}A\ee^{\sigma \wt B}f_{0}\\
k(s,\sigma)&=\partial_{2}E_{B}(\tau-s,f(s))A\partial_{2}E_{B}(s-\sigma,f(\sigma))AE_{B}(\sigma,f_{0})
\end{align*}
we employ a quadrature rule (as in \cite{Jahnke00})
\begin{multline*}
\frac{\tau^{2}}{8}\Bigl(g(0,0)+2g(\tau,0)+g(\tau,\tau)\Bigr)-\int_{0}^{\tau}\!\!\int_{0}^{s}k(s,\sigma)\, \mathrm{d}\sigma \mathrm{d}s \\
=\frac{\tau^{2}}{8}\Bigl(g(0,0)+2g(\tau,0)+g(\tau,\tau)-k(0,0)-2k(\tau,0)-k(\tau,\tau)\Bigr)+d,
\end{multline*}
where $d$ is the remainder term. Consequently, we have to bound
\begin{align*}
g(\tau,\tau)-k(\tau,\tau) &= A^{2}\left(e^{\tau\wt B}f_0-E_{B}(\tau,f_{0})\right), \\
g(0,0)-k(0,0) &= \left(e^{\tau\wt B}-\partial_{2}E_{B}(\tau,f_{0})\right)A^{2}f_0 \\
&= \tau \left(\wt B\varphi_1(\tau \wt B)-R_1(\partial_2 E_B(\cdot,f_0)) \right)A^2 f_0, \\
g(\tau,0)-k(\tau,0)&= A\left( \ee^{\tau\wt B}-\partial_{2}E_{B}(\tau,f_{0})\right) Af_0\\
&= \tau A\left(\wt B\varphi_1(\tau \wt B)-R_1(\partial_2 E_B(\cdot,f_0)) \right)A f_0.
\end{align*}
The first term can again be bounded by using assumption \eqref{eq:ass_1}, now with $i=2$. In addition, we can bound the second and third term using assumption \eqref{eq:ass_3} with $i=1$ and $i=2$ and assumption \eqref{eq:ass_4} with $i=1$ and $i=2$, respectively. Finally, the remainder term depends on the first partial derivatives of $k(s,\sigma)$ and can be bounded by \eqref{eq:ass_5b} and \eqref{eq:ass_5c}.

%
%

\emph{Terms of order $3$}.
In order to bound the remainder terms in the expansion of the exact solution as well as the	approximate solution, we need assumption \eqref{eq:ass_6} and \eqref{eq:ass_7} respectively.
\end{proof}


\section{Convergence analysis for the Vlasov--Poisson equations} \label{sec:vp}

We will consider the Vlasov--Poisson equations in 1+1 dimensions, i.e.
\begin{equation}\label{eq:vp1d}
\left\{
\begin{aligned}
\partial_t f(t,x,v) &= -v\partial_x f(t,x,v) - \mathcal E(f(t,\cdot,\cdot),x) \partial_v f(t,x,v) \\
\partial_x \mathcal E(f(t,\cdot,\cdot),x)   &= \int_{\mathbb{R}} f(t,x,v)\,\mathrm{d}v-1 \\
f(0,x,v)			&= f_0(x,v)
\end{aligned}
\right.
\end{equation}
with periodic boundary conditions in space. For a function $g=g(x,v)$ the abstract differential operators $A$ and $B$ of the previous sections have thus the form
\begin{equation*}
Ag(x,v) = -v\partial_x g(x,v),\qquad Bg(x,v) = - \mathcal E(g,x)\partial_v g(x,v).
\end{equation*}
The domain of interest is given by $(t,x,v)\in [0,T]\times [0,L]\times \mathbb{R}$. Thus, for all $x\in\mathbb{R}$
\begin{equation*}
f(t,x,v)=f(t,x+L,v).
\end{equation*}
By equation~\eqref{eq:vp1d} the electric field $\mathcal E$ is only determined up to a constant. This constant is chosen such that $\mathcal E$ has zero integral mean (electrostatic condition). As will be apparent in the next section it is unnecessary to impose boundary conditions in the velocity direction. This is due to the fact that for a function $f_0$ with compact support in the velocity direction the solution will continue to have compact support for all finite time intervals $[0,T]$ (see Theorem \ref{thm:regularity} below).

\subsection{Definitions and notation}

The purpose of this section is to introduce the notations and mathematical spaces necessary for giving existence,  uniqueness, and regularity results as well as to conduct the estimates necessary for showing consistency and stability.

For the convergence proof we will use the Banach space $L^1([0,L]\times\mathbb{R})$ exclusively. This is reasonable as the solution $f$ of~\eqref{eq:vp1d} represents a probability density function. As such the $L^1$ norm is conserved for the exact (as well as the approximate) solution. Nevertheless, all the estimations could be done as well, for example, in $L^\infty([0,L]\times\mathbb{R})$.

However, we need some regularity of the solution. This can be seen from the assumptions of Theorem~\ref{thm:abstract_consistency}, where we have to apply a number of differential operators to the solution $f(t)$. Thus, we introduce the following spaces of continuously differentiable functions
\begin{eqnarray*}
\mathcal{C}_{\mathrm{per,c}}^m
    &=& \left\{g \in \mathcal{C}^m(\mathbb{R}^2,\mathbb{R})\,;\, \forall x,v \colon (g(x+L,v)=g(x,v)) \land (\text{supp}\; g(x,\cdot) \text{ compact})  \right\}, \\
	\mathcal{C}_{\mathrm{per}}^m
	&=& \left\{g \in \mathcal{C}^m(\mathbb{R},\mathbb{R})\,;\, \forall x \colon g(x+L)=g(x) \right\}.
\end{eqnarray*}	
Together with the norm of uniform convergence of all derivatives up to order $m$, i.e.
\begin{equation*}
\|g\|_{\mathcal{C}_{\mathrm{per,c}}^m} = \sum_{0\le k+\ell\le m} \|\partial_x^k \partial_v^\ell g\|_\infty,\qquad
\|g\|_{\mathcal{C}_{\mathrm{per}}^m} = \sum_{k=0}^m \|\partial_x^k g\|_\infty,
\end{equation*}
the spaces $\mathcal{C}_{\text{per,c}}^m$ and $\mathcal{C}_{\text{per}}$ are turned into Banach spaces.

We also have to consider spaces that involve time. To that end let us define
\begin{eqnarray*}
	\mathcal{C}^m(0,T;C^m)
	&=& \left\{ f \in \mathcal{C}^m([0,T],C^0) \,;\,
	( f(t) \in C^m ) \land
	( \sup_{t\in[0,T]}\left\Vert f(t) \right\Vert_{C^{m}} < \infty )
	\right\},
\end{eqnarray*}
where $C^m$ is taken as either $\mathcal{C}_{\text{per,c}}^m$ or $\mathcal{C}_{\text{per}}^m$. It should be
noted that if it can be shown that the solution $f$ of the Vlasov--Poisson equations lies in the space $\mathcal{C}^m(0,T;C^m)$, we can bound all derivatives (in space) up to order $m$ uniformly in $t\in[0,T]$.


\subsection{Existence, uniqueness, and regularity}

In this section we recall the existence, uniqueness, and regularity results of the Vlasov--Poisson equations in 1+1 dimensions. The theorem is stated with a slightly different notation in \cite{Besse:2008} and \cite{Besse:2005}.

\begin{theorem} \label{thm:regularity}
Assume that $f_0 \in \mathcal{C}_{\mathrm{per,c}}^m$ is non-negative, then $f \in \mathcal{C}^m(0,T;\mathcal{C}_{\mathrm{per,c}}^m)$ and $\mathcal E(f(t,\cdot, \cdot),x)$ as a function of $(t,x)$ lies in $\mathcal{C}^m(0,T;\mathcal{C}_{\mathrm{per}}^m)$. In addition, we can find a number $Q(T)>0$ such that for all $t\in[0,T]$ and $x\in\mathbb{R}$ it holds that $\mathrm{supp} f(t,x,\cdot) \subset [-Q(T),Q(T)]$.
\end{theorem}

\begin{proof}
A proof can be found in \cite{glassey:1996}.
\end{proof}

We also need a regularity result for the electric field that does not directly result from a solution of the Vlasov--Poisson equations, but from some generic function $g$ (e.g., computed from an application of a splitting operator to $f_0$).

\begin{corollary} \label{col:regB}
For $g \in  \mathcal{C}_{\mathrm{per,c}}^m$ it holds that $\mathcal E(g,\cdot)\in \mathcal{C}_{\mathrm{per}}^m$.
\end{corollary}

\begin{proof}
The result follows from the proof of Theorem \ref{thm:regularity}. In addition, in the 1+1 dimensional case it can also be shown easily by starting from the exact representation of the electromagnetic field that is given in \eqref{eq:em_field} below.
\end{proof}

With the arguments contained in the proof of Theorem \ref{thm:regularity}, the regularity results given can be extended to the differential equations generated by $B$ and $\wt B$. Thus, Theorem \ref{thm:regularity} remains valid if $E_B(t,f_0)$ or $\mathrm{e}^{t\wt B}f_0$ is substituted for $f(t)$.

\subsection{Consistency}

The most challenging task in proving the assumptions of Theorem \ref{thm:abstract_consistency} is to control the derivative of $E_B$ with respect to the initial value. The following lemma accomplishes that.

\begin{lemma} \label{le:derivative_initial_value}
The map
\begin{eqnarray*}
\mathcal{C}^{m}_{\mathrm{per,c}}\times\mathcal{C}^{\ell}_{\mathrm{per,c}}&\to& \mathcal{C}^{\min(m-1,\ell)}_{\mathrm{per,c}} \\
(u_0,g)&\mapsto&\partial_{2}E_{B}(t,u_0)g,
\end{eqnarray*}
is well defined.
\end{lemma}

\begin{proof}
We consider $u^{\prime}(t)=Bu(t)$ with $u(0)=u_0$. Motivated by the method of characteristics we can write
\begin{eqnarray*}
V_{u_0}^{\prime}(t)&=&-\mathcal E(u(t,\cdot,\cdot),x) \\
V_{u_0}(0)&=&v \\
u(t,x,v)&=&u_{0}(x,V_{u_0}(t)(x,v)).
\end{eqnarray*}
To show that $V_{u_0}$ depends affinely on the initial value $u_0$, let us integrate $u'(t)=Bu(t)$ with respect to the velocity; this gives at once
\[
	\frac{\mathrm{d}}{\mathrm{d}t} \int_{-\infty}^{\infty} u(t) \,\mathrm{d}v =
	- \mathcal E(u(t,\cdot,\cdot),x) \int_{-\infty}^{\infty} \partial_v u(t) \,\mathrm{d}v
\]
which using integration by parts and the fact that $u(t)$ has compact support (in the velocity direction) shows that the time derivative on the left hand side vanishes. Therefore,
\[
	\mathcal E(u(t,\cdot,\cdot),x) = \mathcal E(u_0,x),
\]
from which the desired result follows.

Computing the G\^{a}teaux derivative with respect to the direction $g$ we get
\begin{align*}
	\partial_{h}E_{B}(t,u_{0}+hg)(x,v)\vert_{h=0} &= \left(\partial_{2}u_{0}\right)\left(x,V_{u_{0}}(t)(x,v)\right)\left(V_{g}(t)(x,v)-v\right) \\ 
	& \qquad+  g(x,V_{u_{0}}(t)(x,v)),
\end{align*}
since $V$ is affine with respect to the initial value. From this representation the result follows.
\end{proof}

%
%
%

The following two lemmas present time derivatives up to order two of $Bf$, $\wt Bf$ and $E_B(t,f_0)$  which follow from a simple calculation. Let us start with the derivatives of the operator $B$ and $\wt B$ applied to the exact solution $f(t)=f(t,\cdot,\cdot)$.

\begin{lemma} \label{le:diffB}
For $f$ sufficiently often continuously differentiable, we have
\begin{align*}
\partial_t Bf(t,x,v)
&= -\mathcal E\left(f^{\prime}(t),x\right)\partial_{v}f(t,x,v) - \mathcal E(f(t),x)\partial_{vt}f(t,x,v) \\
\partial_t^2 Bf(t,x,v)
&= -\mathcal E\left(f^{\prime\prime}(t),x\right)\partial_{v}f(t,x,v) \\ 
&\quad\, - 2\mathcal E\left(f^{\prime}(t),x\right)\partial_{vt}f(t,x,v) - \mathcal E(f(t),x)\partial_{vtt}f(t,x,v) 
\end{align*}
and
\begin{align*}
\partial_t \wt Bf(t,x,v)
&= -\mathcal E\left(f^{\prime}(t+\tfrac{\tau}2),x\right)\partial_{v}f(t,x,v) - \mathcal E(f(t+\tfrac{\tau}2),x)\partial_{vt}f(t,x,v) \\
\partial_t^2 \wt Bf(t,x,v)
&= -\mathcal E\left(f^{\prime\prime}(t+\tfrac{\tau}2),x\right)\partial_{v}f(t,x,v) \\
&\quad\, - 2\mathcal E\left(f^{\prime}(t+\tfrac{\tau}2),x\right)\partial_{vt}f(t,x,v) - \mathcal E(f(t+\tfrac{\tau}2),x)\partial_{vtt}f(t,x,v)
\end{align*}
\end{lemma}
\begin{proof}
From the relations $Bf(t,x,v)=-\mathcal E(f(t),x)\partial_{v}f(t,x,v))$ and $\wt B f(t,x,v) = -\mathcal E(f(t+\tfrac{\tau}2,x)\partial_{v}f(t,x,v)$ the result follows by the product rule.
\end{proof}

Further, we have to compute some derivatives of the evolution operator $E_B(t,f_0)$ with respect to time.

\begin{lemma} \label{le:deriv_deriv_initial_value} For $f$ sufficiently often continuously differentiable, we have
\begin{eqnarray*}
\partial_t E_B(t,f_0) &=& BE_B(t,f_0) \\
&=& - \mathcal E(E_B(t,f_0),\cdot)\partial_v E_B(t,f_0) \\
\partial_t^2 E_B(t,f_0) &=& -\mathcal E(E_B(t,f_0),\cdot) \partial_v (BE_B(t,f_0))
- \mathcal E(BE_B(t,f_0),\cdot) \partial_v E_B(t,f_0) \\
\partial_t (\partial_2E_B(t,f_0)) &=&  -\mathcal E(E_B(t,f_0),\cdot) \partial_v (\partial_2 E_B(t,f_0))
-\mathcal E(\partial_2 E_B(t,f_0),\cdot) \partial_v E_B(t,f_0) \\
\partial_t^2 (\partial_2E_B(t,f_0)) &=&	-\mathcal E(BE_B(t,f_0),\cdot)\partial_v (\partial_2 E_B(t,f_0)) \\
&& {}- \mathcal E(E_B(t,f_0),\cdot)\partial_v(\partial_t(\partial_2 E_B(t,f_0))) \\
&& {}- \mathcal E(\partial_t(\partial_2E_B(t,f_0)),\cdot) \partial_v E_B(t,f_0) \\
&& {}- \mathcal E(\partial_2 E_B(t,f_0),\cdot) \partial_v BE_B(t,f_0).
\end{eqnarray*}
\end{lemma}

\begin{proof}
From the relation $Bf(t,x,v)=-\mathcal E(f(t),x)\partial_{v}f(t,x,v))$ the result follows by a simple calculation.
\end{proof}

It is also necessary to investigate the behavior of the $\varphi$ functions introduced in Definition~\ref{def:varphi}.

\begin{lemma} \label{le:varphi}
For the Vlasov--Poisson equations the functions $\varphi_i(\tau E)$ with
$E\in\{A,\wt B\}$ are maps from $\mathcal{C}^m_{\mathrm{per,c}}$ to $\mathcal{C}^m_{\mathrm{per,c}}$ for all $\tau \geq 0$ and $i\in\mathbb{N}$.
\end{lemma}

\begin{proof} For $i=0$ we have
\begin{equation*}
\ee^{-\tau v\partial_x}f_0(x,v) = f_0(x-\tau v,v),
\end{equation*}
and
\begin{equation*}
\ee^{-\tau \mathcal E\left(f\left(\tfrac{\tau}2\right),x\right)\partial_v}f_0(x,v)=f_0\left(x,v-\tau \mathcal E(f(\tfrac{\tau}2),x)\right).
\end{equation*}
This clearly doesn't change the differentiability properties. 

For the $\varphi$ functions the desired result follows at once from the representation given in~\eqref{eq:varphi}.
\end{proof}

Now we are able to show that all the assumptions of Theorem \ref{thm:abstract_consistency} are fulfilled and that we thus have consistency of order $2$. This is the content of the following theorem.

\begin{theorem} \label{thm:consistency_vp}
Suppose that $f_0 \in \mathcal{C}^3_{\mathrm{per,c}}$ is non-negative. Then the Strang splitting scheme~\eqref{eq:def-strang-one} for the Vlasov--Poisson equations is consistent of order 2.
\end{theorem}

\begin{proof}
	The proof proceeds by noting that the solution has compact support (for a finite time interval), i.e., we can estimate $v$ by some constant $Q$. On the other hand it is clear that for $f_0\in\mathcal{C}^{m+1}_{\mathrm{per,c}}$  we get $Af_0 \in \mathcal{C}^m_{\mathrm{per,c}}$ and $\wt B f_0 \in \mathcal{C}^m_{\mathrm{per,c}}$. The same is true for $B$ as can be seen by Corollary \ref{col:regB}. Therefore, we can establish the bounds \eqref{eq:ass_-1}, \eqref{eq:ass_-0.5}, and \eqref{eq:ass_0}.
Noting that, by Lemma \ref{le:diffB}, terms of the form 
$R_i(\partial_2 E_B)$ are mappings from $\mathcal{C}^{m+i}_{\mathrm{per,c}}$ to $\mathcal{C}^m_{\mathrm{per,c}}$ and that, by Lemma~\ref{le:varphi}, the $\varphi$ functions are mappings from $\mathcal{C}^m_{\mathrm{per,c}}$ to $\mathcal{C}^m_{\mathrm{per,c}}$ we can conclude that after applying all operators in assumptions \eqref{eq:ass_1}, \eqref{eq:ass_2}, \eqref{eq:ass_3}, and \eqref{eq:ass_4} we get a continuous function. By the regularity results we can bound these functions uniformly in time. The same argument also shows the validity of the bound in assumption \eqref{eq:ass_7}.
	
Finally, with the help of Lemmas \ref{le:derivative_initial_value} and \ref{le:deriv_deriv_initial_value} together with the above observations we can verify the bounds in assumptions \eqref{eq:ass_5a}, \eqref{eq:ass_5b}, \eqref{eq:ass_5c}, and \eqref{eq:ass_6}.
\end{proof}

\subsection{Stability}\label{subsec:stability}
We have to verify that the Strang splitting scheme \eqref{eq:def-strang-one} satisfies the conditions \eqref{eq:stability} and \eqref{eq:L-cond}. The stability bound~\eqref{eq:stability} is obviously fulfilled since
\begin{equation*}
\left\Vert \ee^{\frac{\tau}{2}A}\ee^{\tau B_{k+1/2}}\ee^{\frac{\tau}{2}A} f(t)\right\Vert_1 \leq \Vert f(t) \Vert_1.
\end{equation*}
This follows from the proof of Lemma \ref{le:varphi} as the above operators can be represented as translations only (note that a translation does not change the $L^1$ norm).

To verify~\eqref{eq:L-cond}, which can be seen as a substitute for non-linear stability, it remains to be shown that
\begin{equation*}
\Vert g\bigl(x,v-\tau\mathcal E (f_{k+1/2},x)\bigr) - g\bigl(x,v-\tau\mathcal E (f(t_k+\tfrac{\tau}2),x))\bigr) \Vert_1 \le \Vert f_{k+1/2} - f(t_k+\tfrac{\tau}2) \Vert_1
\end{equation*}
for $g(x,v) = \ee^{\frac{\tau}{2}A}f(t_k,x,v) = f(t_k, x-\tfrac{\tau}2 v,v)$. This follows at once from the Lipschitz continuity of both $\ee^{\frac{\tau}{2}A}f$ and $\mathcal E$.

\subsection{Convergence}

We are now in the position to prove second-order convergence of Strang splitting for the Vlasov--Poisson equations in $L^1$. The same result holds literally in $L^\infty$ (or any other $L^p$ space).

\begin{theorem}
Suppose that $f_0 \in \mathcal{C}^3_{\mathrm{per,c}}$ is non-negative and that the auxiliary method~\eqref{eq:midpoint} is first-order consistent and (locally) Lipschitz continuous with respect to its second argument. Then Strang splitting for the Vlasov--Poisson equations is second-order convergent.
\end{theorem}

\begin{proof}
The result follows from Theorem \ref{thm:consistency_vp}, the bounds given in section~\ref{subsec:stability} and Theorem~\ref{thm:convergence-abstract}.
\end{proof}

Note that the two auxiliary methods \eqref{eq:f_h2} and \eqref{eq:f_h2-bis} below are indeed first-order consistent. If they are employed for the computation of $f_{k+1/2}$, the resulting Strang splitting is second-order convergent.

\section{Numerical experiments} \label{sec:implementation}

In this section we present some numerical experiments. Even if we neglect space discretization for the moment, we still have to settle the choice of $f_{k+1/2}$ which has to be a first-order approximation to $f(t_k+\tfrac{\tau}2)$. This can be achieved by Taylor series expansion, interpolation of previously computed values, or by making an additional Lie--Trotter time step of length $\tau/2$. Since we are interested in time integration only, we choose the latter. This method is trivial to implement (once the Strang splitting scheme is implemented) and doesn't suffer from the numerical differentiation problems of a Taylor expansion. Thus, one possible choice would be to use
\begin{equation}
f_{k+1/2} =  \ee^{\frac{\tau}{2}B(f_k)} \ee^{\frac{\tau}{2}A} f_k 	\label{eq:f_h2}
\end{equation}
in our simulations. That this is indeed a first-order approximation follows in the same way as our convergence proof for Strang splitting. We omit the details.

However, since the semigroup generated by $B(f_k)$ can be represented as a translation in velocity (see the proof of Lemma \ref{le:varphi}) and the electric field depends only on the average of the density function with respect to velocity (i.e., it depends only on the charge density), it is possible to drop the first factor in \eqref{eq:f_h2} without affecting the resulting electric field. Consequently, our choice is
\begin{equation}
f_{k+1/2} =  \ee^{\frac{\tau}{2}A} f_k. 	\label{eq:f_h2-bis}
\end{equation}
Since the computation of \eqref{eq:f_h2-bis} is the first step in the Strang splitting algorithm, this leads to a computationally efficient scheme. This scheme is also employed in \cite{Mangeney:2002}, for example. However, no argument why second-order accuracy is retained is given there.

To compute the electric field we will use the following formula (see e.g.~\cite{Besse:2008})
\begin{equation}\label{eq:em_field}
\begin{aligned}
\mathcal E(f(t,\cdot,\cdot),x) &= \int_0^L K(x,y) \left( \int_\mathbb{R} f(t,y,v)\mathrm{d}v - 1 \right) \mathrm{d}y, \\
K(x,y) &= \begin{cases}
\frac{y}{L}-1 & \ 0\leq x < y,\\
\frac{y}{L} & \ y<x\leq L.
\end{cases}
\end{aligned}
\end{equation}
For space discretization we will employ a discontinuous Galerkin method (based on the description given in \cite{Mangeney:2002}). The approximation is of second-order with 80 cells in both the space and velocity direction. In \cite{Mangeney:2002} the coefficients for discretizations up to order 2 are given. However, it is not difficult to employ a computer program to compute the coefficients for methods of arbitrary order.

\subsection{Landau damping}

The Vlasov--Poisson equations in 1+1 dimensions together with the initial value
\begin{equation*}
f_0(x,v)=\frac{1}{\sqrt{2 \pi}} \ee^{-v^2/2} \left( 1+ \alpha\cos(0.5 x)\right),
\end{equation*}
is called Landau damping. For $\alpha = 0.01$ the problem is called linear or weak Landau damping and for $\alpha=0.5$ it is referred to as strong or non-linear Landau damping. As can be seen, for example, in \cite{Filbet:2001,Crouseilles:2011} and \cite{Rossmanith:2011} Landau damping is a popular test problem for Vlasov codes. We solve this problem on the domain $(t,x,v)\in [0,1]\times[0,4\pi]\times [-6,6]$.

\begin{figure}[t]
\begin{center}
	\includegraphics{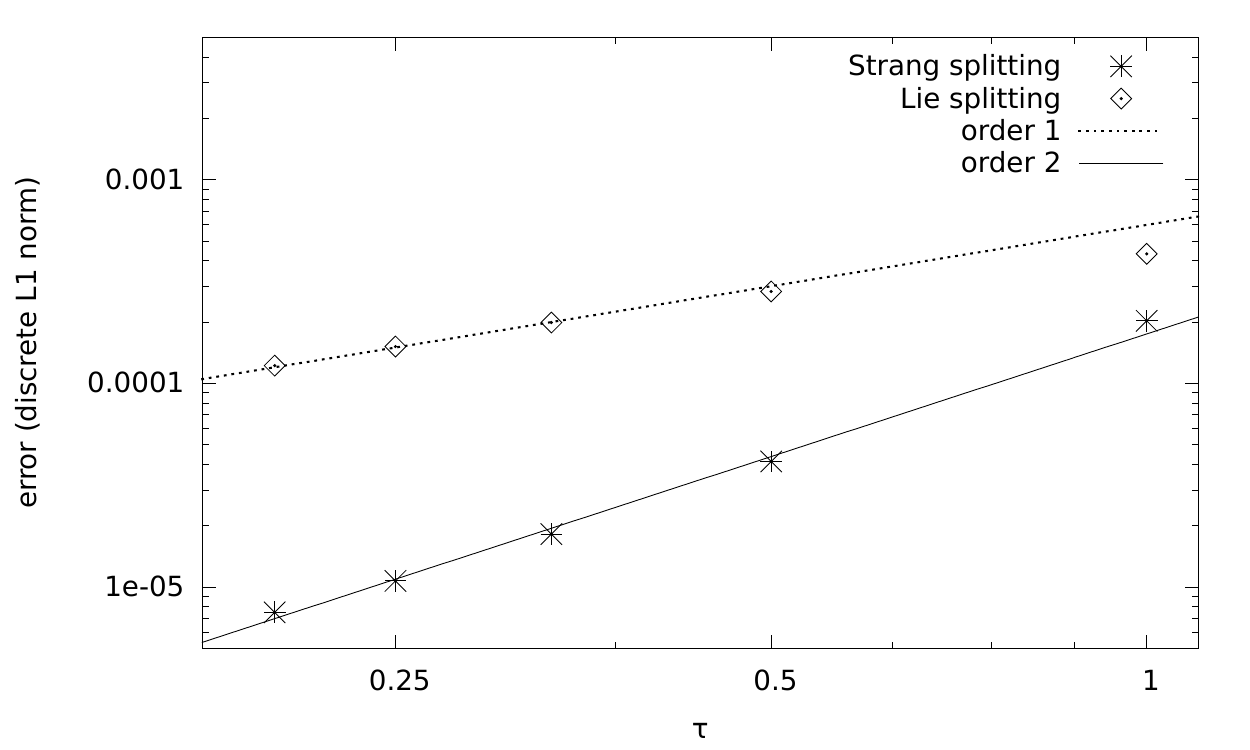}
\end{center}
\begin{center}
	\includegraphics{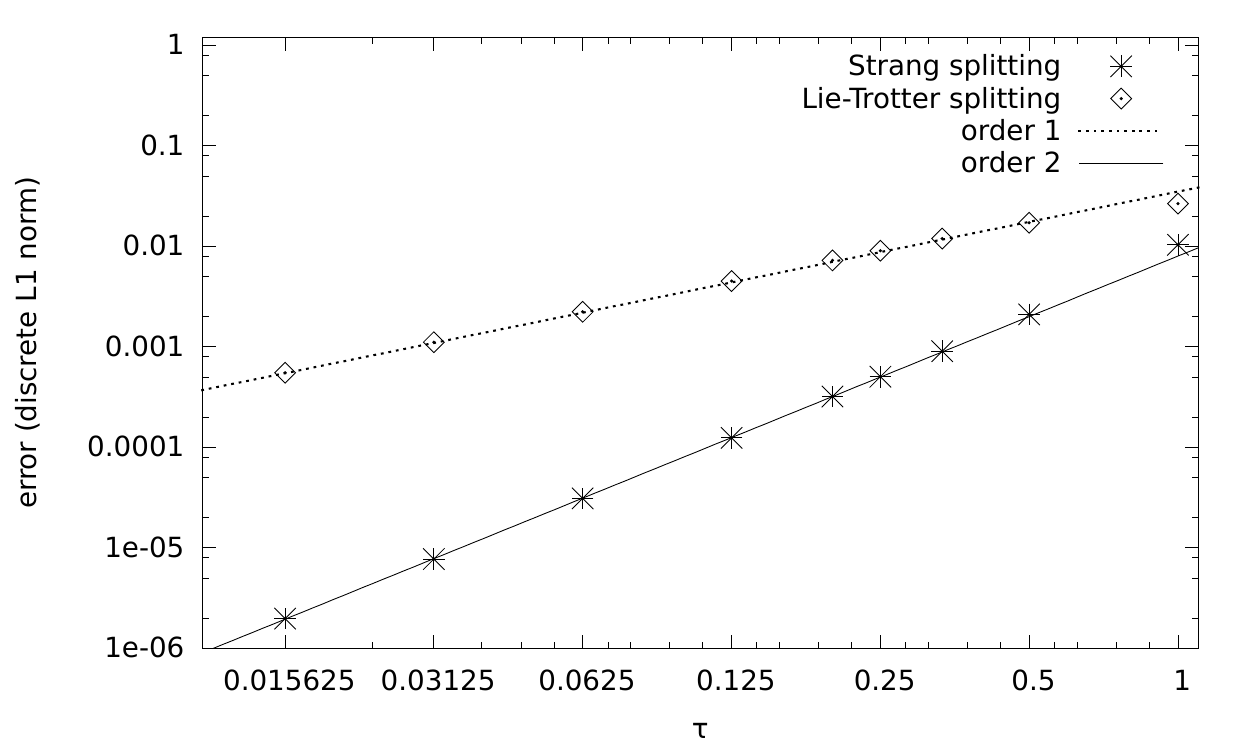}
\end{center}
\caption{Error of the particle density function $f(1,\cdot,\cdot)$ for Strang and Lie--Trotter splitting respectively, where $\alpha=0.01$ (top) and $\alpha=0.5$ (bottom).}
\label{diff}
\end{figure}

For comparison we display the error of the Strang splitting algorithm together with the error for first-order Lie--Trotter splitting. Since we are mainly interested in the time integration error and there is no analytical solution of the problem available, we compare the error for different step sizes with a reference solution computed with $\tau=3.9\cdot 10^{-3}$. The correctness of our code was verified with an upwind scheme on a fine grid with up to 2560 grid points in the $x$- and $v$-direction, respectively. For this experiment, the time step size was determined by the CFL condition to be approximately $\tau=6\cdot 10^{-4}$. The error is computed in the discrete $L^1$ norm at time $t=1$.
The results given in Figure \ref{diff} are in line with the theoretical convergence results derived in this paper.

\section{Conclusion}

In this paper sufficient conditions are given that guarantee convergence of order $2$ for the Strang splitting algorithm in the case of Vlasov-type equations. It is also shown that the Vlasov--Poisson equations in 1+1 dimensions is an example of a Vlasov-type equation, i.e., they fit into the framework of the analysis conducted. For the simulation on a computer, however, a further approximation has to be made (i.e., some sort of space discretization has to be introduced). This approximation is not included in the analysis done here. Nevertheless, the numerical experiments suggest that second-order convergence is retained in the fully discretized case as well.

\section*{Acknowledgments}
The authors thank the referees for providing numerous suggestions that helped to improve the presentation of this paper.

\bibliography{papers}
\bibliographystyle{siam}

\clearpage

\end{document}